\documentclass[12pt]{amsart}
\usepackage{amscd,amssymb,latexsym}
\usepackage{fullpage}
\newcommand{\Mdef}[2]{\newcommand{#1}{\relax \ifmmode #2 \else $#2$\fi}}

\newcommand{\tensor}{\otimes}

\newcommand{\sdr}{\rtimes}

\Mdef{\bhom}{\mathbf{\hat{H}om}}

\Mdef{\Mod}{\mathrm{mod}}

\newcommand{\st}{\; | \;}

\newtheorem{thm}{Theorem}[section]
\newtheorem{lemma}[thm]{Lemma}
\newtheorem{prop}[thm]{Proposition}
\newtheorem{cor}[thm]{Corollary}

\theoremstyle{definition}

\newtheorem{example}[thm]{Example}

\newtheorem{remark}[thm]{Remark}

\newcommand{\qqed}{\qed \\[1ex]}
\renewenvironment{proof}[1][\hspace*{-.8ex}]{\noindent {\bf Proof #1:\;}}{\qqed}

\Mdef{\PH} {\Phi^H}
\Mdef{\PK} {\Phi^K}
\Mdef{\PL} {\Phi^L}
\Mdef{\PT} {\Phi^{\T}}

\Mdef{\ef}{E{\cF}_+}
\Mdef{\etf}{\widetilde{E}{\cF}}
\Mdef{\eg}{E{G}_+}
\Mdef{\etg}{\tilde{E}{G}}
\Mdef{\infl}{\mathrm{inf}}
\Mdef{\defl}{\mathrm{def}}
\Mdef{\res}{\mathrm{res}}
\Mdef{\ind}{\mathrm{ind}}
\Mdef{\coind}{\mathrm{coind}}

\Mdef{\univ}{\mathcal{U}}

\newcommand{\Gr}{\mathrm{Gr}}

\Mdef{\Fp}{\mathbb{F}_p}
\Mdef{\Zpinfty}{\Z /p^{\infty}}
\Mdef{\Zpadic}{\Z_p^{\wedge}}

\newcommand{\lra}{\longrightarrow}

\newcommand{\lr}[1]{\langle #1 \rangle}

\newcommand{\Gspectra}{\mbox{$G$-{\bf spectra}}}

\newcommand{\spec}{\mathrm{Spec}}

\Mdef{\we}{\mathbf{we}}
\Mdef{\fib}{\mathbf{fib}}
\Mdef{\cof}{\mathbf{cof}}
\Mdef{\BI}{\mathcal{BI}}

\Mdef{\B}{\mathbb{B}}
\Mdef{\C}{\mathbb{C}}
\Mdef{\D}{\mathbb{D}}
\Mdef{\E}{\mathbb{E}}
\Mdef{\T}{\mathbb{T}}
\Mdef{\F}{\mathbb{F}}
\Mdef{\G}{\mathbb{G}}
\Mdef{\I}{\mathbb{I}}
\Mdef{\N}{\mathbb{N}}
\Mdef{\Q}{\mathbb{Q}}
\Mdef{\R}{\mathbb{R}}
\Mdef{\bbS}{\mathbb{S}}
\Mdef{\Z}{\mathbb{Z}}

\Mdef{\bA}{\mathbb{A}}
\Mdef{\bB}{\mathbb{B}}
\Mdef{\bC}{\mathbb{C}}
\Mdef{\bD}{\mathbb{D}}
\Mdef{\bE}{\mathbb{E}}
\Mdef{\bF}{\mathbb{F}}
\Mdef{\bG}{\mathbb{G}}
\Mdef{\bH}{\mathbb{H}}
\Mdef{\bI}{\mathbb{I}}
\Mdef{\bJ}{\mathbb{J}}
\Mdef{\bK}{\mathbb{K}}
\Mdef{\bL}{\mathbb{L}}
\Mdef{\bM}{\mathbb{M}}
\Mdef{\bN}{\mathbb{N}}
\Mdef{\bO}{\mathbb{O}}
\Mdef{\bP}{\mathbb{P}}
\Mdef{\bQ}{\mathbb{Q}}
\Mdef{\bR}{\mathbb{R}}
\Mdef{\bS}{\mathbb{S}}
\Mdef{\bT}{\mathbb{T}}
\Mdef{\bU}{\mathbb{U}}
\Mdef{\bV}{\mathbb{V}}
\Mdef{\bW}{\mathbb{W}}
\Mdef{\bX}{\mathbb{X}}
\Mdef{\bY}{\mathbb{Y}}
\Mdef{\bZ}{\mathbb{Z}}

\Mdef{\cA}{\mathcal{A}}
\Mdef{\cB}{\mathcal{B}}
\Mdef{\cC}{\mathcal{C}}
\Mdef{\mcD}{\mathcal{D}} % Something funny about \cD.
\Mdef{\cE}{\mathcal{E}}
\Mdef{\cF}{\mathcal{F}}
\Mdef{\cG}{\mathcal{G}}
\Mdef{\mcH}{\mathcal{H}} % There's something funny about \cH: it 
\Mdef{\cI}{\mathcal{I}}
\Mdef{\cJ}{\mathcal{J}}
\Mdef{\cK}{\mathcal{K}}
\Mdef{\mcL}{\mathcal{L}}% There's something funny about \cL: it 
\Mdef{\cM}{\mathcal{M}}
\Mdef{\cN}{\mathcal{N}}
\Mdef{\cO}{\mathcal{O}}
\Mdef{\cP}{\mathcal{P}}
\Mdef{\cQ}{\mathcal{Q}}
\Mdef{\mcR}{\mathcal{R}}% There's something funny about \cR: it 
\Mdef{\cS}{\mathcal{S}}
\Mdef{\cT}{\mathcal{T}}
\Mdef{\cU}{\mathcal{U}}
\Mdef{\cV}{\mathcal{V}}
\Mdef{\cW}{\mathcal{W}}
\Mdef{\cX}{\mathcal{X}}
\Mdef{\cY}{\mathcal{Y}}
\Mdef{\cZ}{\mathcal{Z}}

\Mdef{\ca}{\mathcal{a}}
\Mdef{\ct}{\mathcal{t}}

\Mdef{\At}{\tilde{A}}
\Mdef{\Bt}{\tilde{B}}
\Mdef{\Ct}{\tilde{C}}
\Mdef{\Et}{\tilde{E}}
\Mdef{\Ht}{\tilde{H}}
\Mdef{\Kt}{\tilde{K}}
\Mdef{\Lt}{\tilde{L}}
\Mdef{\Mt}{\tilde{M}}
\Mdef{\Nt}{\tilde{N}}
\Mdef{\Pt}{\tilde{P}}

\Mdef{\tA}{\tilde{A}}
\Mdef{\tB}{\tilde{B}}
\Mdef{\tC}{\tilde{C}}
\Mdef{\tE}{\tilde{E}}
\Mdef{\tH}{\tilde{H}}
\Mdef{\tK}{\tilde{K}}
\Mdef{\tL}{\tilde{L}}
\Mdef{\tM}{\tilde{M}}
\Mdef{\tN}{\tilde{N}}
\Mdef{\tP}{\tilde{P}}

\Mdef{\ft}{\tilde{f}}
\Mdef{\xt}{\tilde{x}}
\Mdef{\yt}{\tilde{y}}

\Mdef{\Ab}{\overline{A}}
\Mdef{\Bb}{\overline{B}}
\Mdef{\Cb}{\overline{C}}
\Mdef{\Db}{\overline{D}}
\Mdef{\Eb}{\overline{E}}
\Mdef{\Fb}{\overline{F}}
\Mdef{\Gb}{\overline{G}}
\Mdef{\Hb}{\overline{H}}
\Mdef{\Ib}{\overline{I}}
\Mdef{\Jb}{\overline{J}}
\Mdef{\Kb}{\overline{K}}
\Mdef{\Lb}{\overline{L}}
\Mdef{\Mb}{\overline{M}}
\Mdef{\Nb}{\overline{N}}
\Mdef{\Ob}{\overline{O}}
\Mdef{\Pb}{\overline{P}}
\Mdef{\Qb}{\overline{Q}}
\Mdef{\Rb}{\overline{R}}
\Mdef{\Sb}{\overline{S}}
\Mdef{\Tb}{\overline{T}}
\Mdef{\Ub}{\overline{U}}
\Mdef{\Vb}{\overline{V}}
\Mdef{\Wb}{\overline{W}}
\Mdef{\Xb}{\overline{X}}
\Mdef{\Yb}{\overline{Y}}
\Mdef{\Zb}{\overline{Z}}

\Mdef{\db}{\overline{d}}
\Mdef{\hb}{\overline{h}}
\Mdef{\qb}{\overline{q}}
\Mdef{\rb}{\overline{r}}
\Mdef{\tb}{\overline{t}}
\Mdef{\ub}{\overline{u}}
\Mdef{\vb}{\overline{v}}

\Mdef{\hc}{\hat{c}}
\Mdef{\he}{\hat{e}}
\Mdef{\hf}{\hat{f}}
\Mdef{\hA}{\hat{A}}
\Mdef{\hH}{\hat{H}}
\Mdef{\hJ}{\hat{J}}
\Mdef{\hM}{\hat{M}}
\Mdef{\hP}{\hat{P}}
\Mdef{\hQ}{\hat{Q}}

\Mdef{\thetab}{\overline{\theta}}
\Mdef{\phib}{\overline{\phi}}

\Mdef{\uA}{\underline{A}}
\Mdef{\uB}{\underline{B}}
\Mdef{\uC}{\underline{C}}
\Mdef{\uD}{\underline{D}}

\Mdef{\bolda}{\mathbf{a}}
\Mdef{\boldb}{\mathbf{b}}
\Mdef{\bfD}{\mathbf{D}}

\Mdef{\fm}{\frak{m}}
\Mdef{\fp}{\frak{p}}
\newcommand{\fX}{\mathfrak{X}}

\Mdef{\eps}{\epsilon}

\input{xypic}
\usepackage{graphicx}
\newcommand{\sub}{\mathrm{Sub}}

\newcommand{\Qt}{\tilde{\Q}}

\newcommand{\End}{\mathrm{End}}
\renewcommand{\tb}{\overline{\times}}

\newcommand{\diag}{\mathrm{diag}}

\newcommand{\e}{(e)}

\newcommand{\Tt}{\tilde{T}}

\newcommand{\piG}{\pi^G}

\newcommand{\PP}{\mathbb{P}}

\newcommand{\fP}{\mathfrak{P}}
\newcommand{\Rtop}{R^{top}}
\newcommand{\Rttop}{\tilde{R}^{top}}
\newcommand{\Ac}{\check{A}}
\newcommand{\Bc}{\check{B}}
\newcommand{\Abar}{\overline{A}}
\newcommand{\Bbar}{\overline{B}}
\newcommand{\Khat}{\hat{K}}
\newcommand{\Gammat}{\tilde{\Gamma}}

\newcommand{\Mono}{\mathrm{Mono}}
\newcommand{\Aut}{\mathrm{Aut}}
\newcommand{\SSi}{\Sigma}

\begin{document}
\title{The spectral space of conjugacy classes of subgroups of a compact Lie group}

\author{J.P.C.Greenlees}
\address{Mathematics Institute, Zeeman Building, Coventry CV4, 7AL, UK}
\email{john.greenlees@warwick.ac.uk}

\date{}

\begin{abstract}
  The space of conjugacy classes of subgroups of a compact Lie groups
  with its Zariski topology splits as a disjoint union of clopen
  blocks, each dominated by a subgroup $H$ with finite Weyl group, and
  the structure of each block is controlled by an integral
  representation of $\pi_0(H)$.
\end{abstract}

\thanks{  The
  work is partially supported by EPSRC Grant
  EP/W036320/1. The author  would also  like to thank the Isaac Newton
  Institute for Mathematical Sciences, Cambridge, for support and
  hospitality during the programme Equivariant homotopy theory in
  context, where later parts of  work on this paper was undertaken.
  This work was supported by EPSRC grant EP/Z000580/1.  }

\maketitle

\tableofcontents

\section{Introduction}
\subsection{Overview}
The space $\fX_G =\sub(G)/G$ of conjugacy classes of closed subgroups
of $G$ has two topologies of interest.
First of all, it has the h-topology, which is the quotient topology of the
Hausdorff metric topology on $\sub (G)$. This is a Stone space (compact, Hausdorff,
totally disconnected). The space also has a Zariski topology whose
closed sets are the h-closed sets which are also closed under passage
to cotoral subgroups\footnote{$L$ is {\em cotoral} in $K$
if $L$ is normal in $K$ with quotient being a torus. For conjugacy
classes $(L)$ is cotoral in $(K)$ if the relation holds for some
representative subgroups.}. Our 
focus is on the Zariski topology. 

In fact $\fX_G$ (with the Zariski topology) is the Balmer spectrum of
finite rational $G$-spectra \cite{spcgq,prismatic} making it a
spectral space. This also makes it  the basis of an algebraic model for rational 
$G$-spectra. We will briefly describe the intended applications 
to $G$-equivariant cohomology theories in Subsection \ref{subsec:Gspectra}, 
but this paper is purely about the space $\fX_G$ itself.

It was shown in \cite{prismatic} that the space $\fX_G$ has finite
Cantor-Bendixson rank (equal to the rank of $G$) and that the
Cantor-Bendixson and Thomason filtrations agree. The results of the
present paper refine this numerical information to give a rather
detailed description of the space.

We will show (Theorem \ref{thm:partition}) that  the space $\fX_G$ decomposes as a finite union
of Zariski clopen subspaces (which we refer to as `blocks')
$$\fX_G=\sub(G)/G=\cV^G_{H_1}\amalg \cdots \amalg \cV^G_{H_n},$$
where $H_1, H_2, \cdots , H_n$ are subgroups with finite Weyl
group. The block $\cV^G_H$ is dominated by a subgroup $H$ in the sense
that all conjugacy classes are represented by subgroups of $H$, and
the structure of $\cV^G_H$ is controlled by the subgroup $H$.

\subsection{Single blocks}
Here we give an outline of the structure of a single block
$\cV^G_H$ in the sense that we describe the block `up to finite error' to
give a crude guide to the appearance. The body of the paper will make this more precise.
\subsubsection{Shape}
We now describe the structure of an individual block $\cV^G_H$: up to
finite error we have
$$\cV^G_H \sim \fP^G_H\times \cN^G_H$$
where $\fP^G_H$ is the poset of subgroups of a torus with the cotoral
ordering and $\cN^G_H$ is a Stone space which, again up to finite
error, has a product decomposition 
$$\cN^G_H \sim \cN^G_{H,2}\times \cdots \times \cN^G_{H,t}$$
as a product of standard Stone spaces (`integral Grassmannians'). We
outline the shape of these below. Altogether 
$$\cV^G_H \sim \fP^G_H\times \cN^G_{H,2}\times \cdots \times
\cN^G_{H,t}, ;$$
its Cantor-Bendixson rank is finite and equal to the sum of those of
the factors. 

\subsubsection{Fusion}
First of all, the block $\cV^G_H$ for $G$ is obtained from the block $\cV^H_H$ by passing
to $G$-conjugacy (the map $\cV^H_H\lra \cV^G_H$ is a quotient map with
finite fibres of bounded size). We therefore get the general form by looking at all top
blocks $\cV^G_G$.

We first describe the two simplest types of block and then explain how
to mix them. 

\subsubsection{Weyl-finite blocks}
The simplest case is when the block $\cV^G_G$ is
Weyl-finite in the sense that all elements of $\cV^G_G$ have finite Weyl group.
The reason for the simplicity is that in this case $\cV^G_G$ is itself
a Stone space (with trivial cotoral ordering). An example to bear in mind is $G=H=O(2)$, then
$$\cV^G_G=\{ (D_{2n})\st n\geq 1\} \cup \{ O(2)\}, $$
topologised as the one point compactification of the positive
integers. 

\subsubsection{Noetherian blocks}
The next case is when the subgroup $G$ is a finite central extension
of a torus. In this case the topology on $\cV^G_G$ is determined by
the the partially ordered set  of subgroups of the identity
component $G_e$. The first example  is when $G=H$ is itself a circle, and
$$\cV^G_G=\{ (C_n)\st n\geq 1\}\cup \{ SO(2)\}, $$
topologised like $\spec (\Z)$: the proper closed
subspaces are precisely the finite sets of finite subgroups $C_n$.

\subsubsection{Mixed blocks}
For a general block, we mix the behaviour of Weyl-finite and
Noetherian blocks. If $G_e=\Sigma\times_Z T$ where $\Sigma$ is
semisimple, $T$ is a torus
and $Z$ is a finite central subgroup then $T$ is the identity
component of the centre of the identity component, so we call it the
{\em central torus} of $G_e$ and write $T=T_zG$. We can
identify the underlying spectral space $\cV^G_G$ in terms of
the integral representation of the component group $G_d$ on
$\Lambda_0=H_1(T)$, together 
with the extension class $\eps (G)\in H^2(G_d; T)\cong H^3(G_d; \Lambda_0)$.

Up to finite multiplicity, $\cV^G_G$ is a product of a Stone space $\cN^G_G$
space and a poset $\fP^G_G$. The Stone space $\cN^G_G$ is a
neighbourhood of $G$ in the space $\Phi G$ of conjugacy classes of
subgroups with finite Weyl groups. The poset $\fP^G_G$ is obtained
from $\sub (T_z^1G)$ by passing to $G$-conjugacy classes, where $T_z^1G$
is  the central torus of $G$ itself (i.e., the subtorus of $T_zG$ on
which $G_d$ acts trivially).
In fact there is a continuous map 
$$\lambda : \cV^G_G\lra \fP^G_G\times \cN^G_G$$
which is an isomorphism up to finite indeterminacy. We take $\lambda
((K))=( K\cap T_z^1G, \omega(K))$, where $\omega (K)/K$ is 
the conjugacy class of the maximal torus of $W_G(K)$.
The map $\lambda$ is almost surjective in a sense we will explain, and
the fibres over each pair are finite and of bounded cardinality. 

\subsection{The isotypical decomposition of the Stone space}
We may decompose further. Suppose
$$\Lambda_0\tensor \Q =H_1(T; \Q)=S_1^{\oplus m_1}\oplus S_2^{\oplus m_2}\oplus \cdots
\oplus S_t^{\oplus m_t}$$
where $S_1, S_2, \ldots , S_t$ are distinct simple modules with $S_1$
having trivial action. The factor $S_1^{\oplus m_1}$ gives rise to the Noetherian factor
$\fP^G_G$ above.  The factor $\cN^G_{G,S_i}$ corresponds to
$S_i^{\oplus m_i}$, and is a Stone space of Cantor-Bendixson rank
$m_i$, which is, up to finite error, an $m_i$-stage compactification of the
spaces of full rank $W$-sublattices of $\Z^{\oplus m_i}$.

\subsection{Application to $G$-spectra}
\label{subsec:Gspectra}
This subsection explains the motivation for this paper. It
will play no mathematical role, so can be ignored by incurious readers.

If we decompose the space $\sub(G)/G$ as a union of Zariski clopen
subsets
$$\sub(G)/G=\fX_1\amalg \cdots \amalg \fX_n,$$
then since the rational Burnside ring is $[S^0,S^0]^G=A(G)=C(\fX_G,
\Q)$,  idempotents in the Burnside ring show that the category of
rational $G$-spectra decomposes into blocks
$$\Gspectra \simeq \Gspectra \lr{\fX_1}\times  \cdots \times \Gspectra
\lr{\fX_n}$$
where  $\Gspectra \lr{\fX}$ is the category of $G$-spectra with
geometric isotropy in $\fX$.  Thus we may conduct the analysis one
block at a time. 

It is conjectured 
\cite{AGconj} that 
there is a small and calculable abelian model $\cA (G|\fX)$ so that the category DG-$\cA 
(G|\fX)$ of differential graded objects in $\cA (G|\fX)$ is a model 
for $\Gspectra\lr{\fX}$ in the sense that there is a Quillen equivalence 
$$\Gspectra\lr{\fX}\simeq DG-\cA (G|\fX).$$
In general terms $\cA (G|\fX)$ is a category of equivariant sheaves
over $\fX$, so our description of $\cV^G_H$ gives the general form of
$\cA (G|\cV^G_H)$.

\subsubsection{The algebraic model}
In general terms we need two additional pieces of data on the space $\cV^G_H$ to
form an algebraic model $\cA (G|\cV^G_H)$ for rational $G$-spectra
with geometric isotropy in $\cV^G_H$.
\begin{itemize}
\item A sheaf of rings on
$\cV^G_H$ (the stalks over $K$ is the polynomial ring
$H^*(BW_G^e(K))$); the sheaf of rings has additional
structure enabling us to encode the Localization Theorem.
\item A component structure (this associates a finite
group $\cW_K=\pi_0(W_G(K))$ to each point $K\in \fX_G$). 
\end{itemize}

Without specifying all the structure we can outline the strategy for
giving an algebraic model. 
As  described above, it is natural to work block by block.

\subsubsection{Weyl-finite blocks}
We describe the strategy  in the case of Weyl-finite blocks from
\cite{gqwf} as a template. Suppose then that $\fX$ is a Weyl-finite block. 
First, we identify the model in an artificially simple case,
as if all the Weyl groups were connected. The model is then simply
the category of sheaves over $\fX$. We then add in a {\em component
  structure}, giving an action a finite Weyl group acts at each
point; this means that the associated group rings together to
make a sheaf $\Q[\cW]$ over $\fX$. The model  (``equivariant sheaves'') consists of 
sheaves of modules over the sheaf $\Q[\cW]$.

Given this algebraic template, one makes corresponding constructions in the category of
spectra. In both algebra and topology the
category of interest is a weighted limit of module categories over simpler rings.
At each stalk this amounts to considering modules over $\Q$ in
free-$W$-spectra for a finite group $W$, and we need to remember that
the category is generated by the single module $\Q [W]$.

Finally,  it remains to prove that the algebraic and topological
categories are equivalent. By Morita theory, the algebraic and topological categories
are each equivalent to the endomorphism rings of their
generators. In the Weyl-finite case, the endomorphism rings of
generators have homotopy in a single degree, so they are formal, and 
 by Shipley's Theorem \cite{ShipleyHZ}, they are
equivalent to categories of modules over a ring. It just
remains to calculate the homotopy of the endomorphism rings. The
rings arising in topology and algebra are isomorphic and formal,
so the categories of spectra and of algebraic objects are equivalent.  

\subsubsection{General blocks}
For a general block we are guided by the map 
$$\lambda: \cV^G_H \lra  \sub(T^1_z)_G\times \cN^G_H.$$
We use the Weyl-finite model in the $\cN^G_H$ direction and the
Noetherian model in the $\sub(T_1)_G$ direction. More precisely an
object of the algebraic model is an equivariant sheaf over $\cN^G_H$
with the fibre over a subgroup
$K$ being very much like an object of the algebraic category $\cA (T^1_z)$, but with
the poset $\sub(T^1_z)$ replaced by the poset $\omega^{-1}(K)$ (varying
with $K$).

\subsection{Associated work in preparation}
As described above, this paper is part of a programme to give an algebraic model for
rational $G$-spectra for all compact Lie groups $G$. This paper
describes the group theoretic data that feeds into the construction of an
abelian category $\cA (G)$ for all compact Lie groups $G$. It builds
on \cite{t2wqalg} which does this specifically for groups with
identity component a torus.

The path from an understanding of the space $\fX_G$ of subgroups to an
algebraic model for rational $G$-spectra involves several steps. These
have been implemented in the special case that $G=SU(3)$ in the series \cite{gq1,
  t2wqmixed, u2q, su3q}. This gives a good picture of the general
pattern, except that the use of flags of subgroups was unnecessary;
cases where flags are necessary can be seen in \cite{AGtoral,gtoralq}. 

The easiest case of the general programme is that of Weyl-finite
blocks, completed in \cite{gqwf}.  The case of Noetherian blocks is
closely modelled on those of the torus, and an account is in
preparation in \cite{AGnoeth}. An explicit algebraic model for a
general block will be given  in \cite{AVmodel}, and it is hoped that
this will be the basis of the proof that the
category of rational $G$-spectra has an algebraic model in general.

\subsection{Organization}
The rest of the paper is layed out as follows. In Section
\ref{sec:partition} we show that the space $\fX_G$ of subgroups is a disjoint
union of blocks $\cV^G_H$ each dominated by a subgroup $H$ with finite
Weyl group.  In Section \ref{sec:subgroups} we introduce notation for
the features of the subgroup $H$ used in the analysis.
 In Section \ref{sec:flattening} we show each block can be helpfully
 be viewed as lying over the product of a poset and a Stone space. 
Section \ref{sec:blocktop} introduces notation for a basis
 for the topology, and in Section \ref{sec:topology} we describe the
 topology in terms of the toral lattice. In Section \ref{sec:GrassQ},
 we explain how to decompose the Stone space factor up to finite error
 as a product of standard factors which are compactifications of the
 set of sublattices. Finally, in Section \ref{sec:normalizers} we
 establish the upper semicontinuity result for normalizers that is
 used in Section \ref{sec:partition}.

\section{The partition}
\label{sec:partition}
In this section, we will show that the space $\fX_G=\sub(G)/G$ may be decomposed into
finitely many blocks $\cV^G_H$, each dominated by a subgroup $H$ with
finite Weyl group. Later we show that each such block takes a standard form.

\subsection{Strategy}
The subspace $\cV^G_H$ is constructed as follows. We start with
the space $\Phi H$ of $H$-conjugacy classes of subgroups of $H$ with
finite Weyl groups. This  is a compact, Hausdorff, totally disconnected space,
so we may choose compact open neighbourhood $\cN^H_H$ of $H$ in $\Phi H$.
We will show (Lemma \ref{lem:PhiHPhiG}) that $\cN^H_H$ may be
taken small enough that its image $\cN^G_H$ in $\sub(G)/G$ consists of subgroups
with finite Weyl group. Now we take $\cV^G_H=\Lambda_{ct}(\cN^G_H)$ to be the set of
conjugacy classes cotoral in an element of $\cN^G_H$, and we will show
this is clopen in the Zariski topology. 

Thus all conjugacy classes in $\cV^G_H$ are represented by subgroups
$K$ of $H$ with $K$ cotoral in a subgroup $H'$ of $H$ close to
$H$. This is the first justification for saying the block is
`dominated' by $H$, and we will later see how the structure of
$\cV^G_H$ is also controlled by $H$.

\subsection{Three continuity lemmas}

\begin{lemma}
The functor $K\longmapsto K_{\e}:=K\cap H_e$ on subgroups of $H$ is continuous. 
\end{lemma}

\begin{proof}
We need to check that if $\Gamma^i\lra \Gamma^*$ then this is also 
true of the parts inside $H_e$. This is clear since we may find a 
number $\epsilon >0$ so that every point of $H$ outside $H_e$ has 
distance $>\epsilon$ from $H_e$. 
\end{proof}

Similarly, if $\Gamma\subseteq H_e=S\times_ZT$ we may consider the 
inverse image $\Gammat$ under the map $p: S\times T=\Ht \lra 
H=S\times_ZT$. 

\begin{lemma}
The functor $\tilde{(\cdot)}$ on subgroups of $H_e$  is continuous. \qqed 
\end{lemma}

Any subgroup $K$ is cotoral in a subgroup $\Khat$ with finite Weyl 
group (the inverse image of the maximal torus of $W_G(K)$), and this 
is unique up to conjugacy, allowing us to define $\omega (K)=\omega 
(\Khat)$. We recall the perennially useful result \cite{ratmack}
stating that $\omega $ is continuous. 

\begin{lemma}
  \label{lem:omegacts}
The functor $\omega : \sub (H)/H\lra \Phi H$ on subgroups   is continuous. \qqed 
\end{lemma}

\subsection{Fusion}
We show that provided the neighbourhood $\cN_H^H$ of $H$ in $\Phi
H$ small enough, its image in $\sub
(G)/G$ is a neighbourhood of  $(H)$ in $\Phi G$. 
We start from the map  $\sub (H)\lra \sub(G)$,  and consider the composite 
$$\Phi H\lra \sub (H)/H\lra \sub(G)/G.$$

\begin{lemma}
\label{lem:PhiHPhiG}
If $H$ has finite Weyl group then the singleton $(H)_H$ has a neighbourhood $\cN_H^H$
in $\Phi H$ which maps into $\Phi G$. Furthermore the image of 
$\cN_H^H$ contains a neighbourhood $\cN^G_H$ of $(H)_G$ in $\Phi 
G$. 
  \end{lemma}

\begin{proof}
Since $N_H(K)=N_G(K)\cap H$, if $K$ has finite $G$-Weyl group it has
finite $H$-Weyl group. It therefore suffices to show that if $H$ has
finite $G$-Weyl group it has a neighbourhood with the same property. 

By the upper semicontinuity of the normalizer (Theorem \ref{thm:normalizersusc}), we may choose a
neighbourhood of $H$ so that for all $K$ in the neighbourhood
$N_G(K)\subseteq N_G(H)$. 

We have containments $N_G(K)\supseteq N_H(K)\supseteq K$. The second
subquotient is $W_H(K)$. 

Now if $N_G(K)\subseteq N_G(H)$, this induces an embedding
$$\frac{N_G(K)}{N_H(K)}= \frac{N_G(K)}{N_G(K)\cap H} \subseteq 
\frac{N_G(H)}{H}=W_G(H). $$
\end{proof}

\subsection{Cotoral specialization}
We show that taking cotoral closures has good topological properties. 

  \begin{lemma}
    \label{lem:VHclopen}
If $H$ has finite Weyl group and $\cN_H^H$ is chosen so that $\cN^G_H$
consists of subgroups with finite Weyl group then the set
$\cV^G_H:=\Lambda_{ct} \cN^G_H$ is clopen in the constructible topology and closed under 
cotoral specialization.  
    \end{lemma}

    \begin{proof}
Closure under cotoral specialization is true by definition, and 
since the cotoral  condition is transitive, elements of $\cV_H$
      are characterised as being cotoral in elements of $\cN^G_H$.
      
 First, let us see that $\cV_H$ is closed, i.e., that the condition of
 being cotoral in an element of $\cN^G_H$ is a closed condition.  Suppose
 that  $\Gamma^i$ is cotoral in $H_i$ for $(H_i)\in \cN^G_H$ and
 $\Gamma^i\lra \Gamma^*$. Since $\cN^G_H$ is closed under cotoral
 generalization, we may suppose $H_i=\omega (\Gamma_i)$.
 Now by continuity of $\omega$, we deduce
 $H_i=\omega (\Gamma^i)\lra \omega (\Gamma^*)=:H^*$.
 Since $\cN^G_H$ is $h$-closed, $H^*\in \cN^G_H$ and $\Gamma^*$ lies
 in $\cV_H$ as required.
 
On the other hand, to see $\cV^G_H=\Lambda_{ct} \cN^G_H$ is open we need only show that 
if $\Gamma'$ is cotoral in $H'\in \cN^G_H$ then there is a neighbourhood of 
$\Gamma'$ entirely in $\Lambda_{ct} \cN^G_H$. Otherwise there is a sequence of 
points $\Gamma_i$ approaching $\Gamma'$ that are not cotoral in $\cN^G_H$. 
However $\omega (\Gamma_i)\lra \omega (\Gamma')\in \cN$, which is impossible
because $\cN^G_H$ is open. 
\end{proof}

\subsection{Constructing the partition}
We describe here how to form the partitition of
$\sub(G)/G$. In effect it is just a greedy algorithm, repeatedly
taking a neighbourhood of one of the largest remaining subgroups.
We need to prove the process terminates and that it has the
topological properties required.

\begin{thm}
  \label{thm:partition}
The Balmer spectrum $\sub(G)/G$ can be partitioned into finitely many
sets of the form $\cV_H$ (i.e.,  closures under cotoral specialization
of standard Hausdorff clopen neighbourhoods of subgroups).  
\end{thm}

\begin{proof}
  Take $R_0=\sub(G)/G$, and then form a
 sequence of clopen sets
 $$\sub(G)/G=R_0\supset R_1\supset R_2\supset \cdots \supset R_r\supset
 R_{r+1}=\emptyset$$
one step at a time. 
Suppose we have defined $R_0, \ldots, R_s$.  %By \cite[8.4]{spcgq},
The set of cotorally maximal elements of $R_s$ is $h$-closed (if
$x_s\lra x_*$ in $R_s$ with $x_s$ maximal, then by Lemma
\ref{lem:omegacts},  $x_s=\omega (x_s)\lra
\omega (x_*)$, and $x_*=\omega (x_*)$). Accordingly, the set $\max
(R_s)$ of cotorally maximal elements of $R_s$ is $h$-compact. We
partition $\max(R_s)$ into finitely many clopen sets as follows.
The set of conjugacy classes is countable, so we choose an enumeration
$H'_1, H'_2, \ldots$, and then choose neighbourhoods of the groups in
turn. First $H_s^1=H'_1$, and we choose a standard clopen neighbourhood
$\cN_s^1$ inside $\max (R_s)$. At the $j$th stage we will have $H_s^1,
\ldots , H_s^j$ and disjoint clopen neighbourhoods $\cN_s^i$ of
$H_s^i$. Now let $\cN_s^{\leq j}=\bigcup_{i\leq j}\cN^i$ and
 take $H_s^{j+1}=H'_\alpha$ where  $H'_\alpha$ is the first of the
 enumerated subgroups not already contained in
$\cN_s^{\leq j}$. Now let $\cN_s^{j+1}$ to be a standard clopen
neighbourhood inside $\max(R_s)$
and disjoint from $\cN_s^{\leq j}$; this exists because the topology
on a Stone space is generated by clopen sets. The sets $\cN_s^{\leq j}$ cover
$\max(R_s)$ since $H'_i$ is certainly covered by the $i$th stage, and
by compactness this has a finite subcover. Hence $\max (R_s)$ is
covered by $\cN_s^1, \ldots , \cN_s^{n(s)}$ for some $n(s)$.

 By closing these neighbourhoods $\cN_s^j$ under cotoral
specialization, we obtain sets
$\cV_{H_s^1}, \ldots , \cV_{H_s^{n(s)}}$,
each of which is clopen by Lemma \ref{lem:VHclopen}. Next, we argue
that these sets are  disjoint from each other and also from $H_i^j$ with
$i<s$. Suppose 
that $\cV_{H_i^j}\cap \cV_{H_s^k}\neq \emptyset$ with $i\leq s$. Then
there is a subgroup $K$ cotoral both in  a subgroup $H_i^j\in \cN_i^j$ and in a subgroup $H_s^k\in
\cN_s^k$. After conjugation we may suppose $H_i^j, H_s^k\leq N_G(K)$, and 
$H_i^j/K$ and $H_s^k/K$ are two subtori of $W_G(K)$. By uniqueness of 
maximal tori up to conjugation  this shows that up to conjugation 
$H_i^j, H_s^k$ are cotoral in a larger group, contradicting the fact
that $H_s^k$ is not in $\cV_{H_i^j}$ for $i < s$. 

This completes the work for $R_s$ and we then take
$$R_{s+1}=R_s\setminus \bigcup_{j=1}^{n(s)}\cV_{H_s^j}, $$
which is still clopen since the sets $\cV_{H_s^j}$ are clopen.
Cotorally maximal subgroups are of height $\leq r$, so the subgroups
$H_s^j$ are of height $\leq r-s$, and so $R_{r+1}=\emptyset$ as
required.
\end{proof}

\section{The structure of subgroups}
\label{sec:subgroups}
In preparation for studying the space of subgroups, we describe the 
structure of individual subgroups of $H$.

\subsection{Some standard notation}
For any  compact Lie group $H$, there is a short exact sequence 
$$1\lra H_e\lra H \lra H_d\lra 1. $$
We say that a subgroup $\Gamma$ of $H$ is {\em full} if it maps onto 
$H_d$. 
%We also write $\Gamma_{\e}=H_e\cap \Gamma$ for the part of 
%$\Gamma$ inside the identity component of $H$. The identity 
%component of $\Gamma$ will lie inside $\Gamma_{\e}$, but usually
%$\Gamma_{\e}$ will be bigger. 

By the structure theorem for compact connected Lie groups, 
$$H_e=\Sigma \times_Z T$$
with $\Sigma $ semisimple, $T$ a torus and $Z$ a finite central 
subgroup.

\begin{lemma}
\label{lem:SigmaTchar}
  The images  of $T$ and $\Sigma$ in $H_e$ are characteristic, and
  hence normal and invariant under the action of 
$H_d$. 
\end{lemma}

\begin{proof}
The subgroup $T$ is the identity component of the centre of the
identity component.

We note $\Sigma$ is normal in $H_e=\Sigma \times_Z T$, and the
quotient is a torus. Any map from a semisimple group $\Sigma'$ to a torus is
trivial, and hence the image of any map  $\Sigma'\lra H_e$ lies in $\Sigma$.
 \end{proof}

 We write $T=T_zH$ and call it the {\em central torus} of $G_e$.

The {\em toral lattice} is the integral representation of $H_d$ on $\Lambda_0:=H_1(T; \Z)$ is the most
important invariant of $H$ as far as understanding the top block is
concerned. Rationalizing,   $H_d$ has a rational  represention $V=H_1(T; \Q)$, with 
   isotypical decomposition  $V=V_1\oplus \cdots \oplus V_s$.
Now the exponential map gives an isomorphism $V^{\R}/\Lambda_0 \cong T$. We may restrict the exponential
map to $V_i^{\R}$, and obtain an embedding
$$\Tt_i:=V_i^{\R}/\Lambda_i\lra T$$
where $\Lambda_i=\Lambda\cap V_i$. However the combined map
$$\Tt=\Tt_1\times \cdots \times \Tt_z\lra T$$
may be a non-trivial finite cover (because $\Sigma_i \Lambda_i $ is of
rank $z$, but may be a proper subgroup of $\Lambda_0$); we write $L$ for
its kernel.

We will often have to treat the trivial representation $V_1$
differently from the rest, so we write
$$V_f:=V_2\oplus \cdots \oplus V_s, \Lambda_f:=\Lambda\cap V_f,
\Tt_f:=V_f^{\R}/\Lambda_f$$
and 
$$L_f:=\ker (\Tt_1\times \Tt_f \lra T).$$
\begin{remark} 
We note that the finite subgroup $L_f$ may be non-trivial.  The
simplest example occurs when $H$ is the semidirect product of $T^2$ by the group of
order 2 exchanging factors when $L_f=\{ (1, 1), (-1,-1)\}$.
\end{remark}

\begin{remark}
When dealing with non-trivial kernels, we may proceed as follows. 
If $C$ is an $H_d$-invariant subgroup of $T$ then  we take $\Ac
=C^{H_d}=C\cap \Tt_1, \Bc=C\cap \Tt_f$ and let bars denote
the projections of $C$ into $\Tt_1$ and $\Tt_f$. We then have product bounds
for $\tilde{C}$ and the consequent product bounds for $C$:
$\Ac \times \Bc \subseteq \tilde{C} \subseteq \Abar\times 
\Bbar$.

If $(a,b)$ lies in a $H$-invariant subgroup $C$, so does 
 $g(a,b)=(a, gb)$, and hence $(0, (1-g)b)$ for all $g\in H$, 
so  we have 
$$J\Bbar=\{ (g-g')b\st b\in \Bbar, g, g'\in H_d\}\subseteq
\Bc\subseteq \Bb$$
Thus $\Bbar/\Bc$ is a quotient of  $H_0(W;
\Bbar)=\Bbar /J\Bbar$ and therefore  of finite order with a trivial action of $H_d$. If it has exponent $n$ then
 $\Abar/\Ac$ is also annihilated by $n$ and hence of finite order. In the favourable case that
 $\Bc/J\Bbar=0$ then $\Bc=\Bbar$ and $\Ac=\Abar$. 
\end{remark}

The following lemma gives some crude but useful information. 

\begin{lemma}
  Suppose  $K$ is a full subgroup of $H$ with 
  $K\cap H_e=\Sigma\times_Z C$ (all subgroups $K$ close enough to $H$  have this property). 

  (i) If $K$ has finite $H$-Weyl group then $C\supseteq T_1$ 
and hence 
  $C=(\Tt_1\times B)/L_f$ where $B\subseteq 
\Tt_f$ and $L_f$ is  the finite subgroup described above.

  (ii) If $K$ is cotoral in $H$ then $K\supseteq T_f$,   and hence 
  $C=(A\times \Tt_f )/L_f$ where $A\subseteq 
\Tt_1$ and $L_f$ is  the finite subgroup described above.  
\end{lemma}

\begin{proof}
We have seen that $K=H(C, \sigma)$, and $N_H(K)=H(C^+, \sigma)$ where 
$C^+=\{t\st t^{-1}t^w\in C \mbox{ whenever } w\in H_d\}$.

If $K$ has finite Weyl group, this shows 
$T_1\subseteq T^{H_d}\subseteq C^+$. If $C\cap T_1\not = T_1$ then $T_1/C\cap 
T_1$ (and hence also $W_H(K)$)  is infinite. Thus $T_1\subseteq C$ and 
$\tilde{C}=\Tt_1\times B$ as required. 

If $K$ is cotoral in $H$  then since $C^+=T$. If we write $H_1(C;\Q)$
as a sum of isotypical pieces $W_\alpha$ then $W_\alpha =V_\alpha$ for all
$\alpha\neq 1$. Hence $T_f \subseteq C$ and $\tilde{C}=A\times \Tt_f$
as required. 
\end{proof}

\section{The structure of a block}
\label{sec:flattening}
We first consider the  top block $\cV^H_H$. It is finite over a space
defined by  the toral lattice $\Lambda_0=H_1(T_zH)$ as an 
integral representation of $H_d=\pi_0(H)$.  Using this finite map, we
can deduce the correponding statement for all blocks $\cV^G_H$. 
Discussion of the topology is deferred to Section
\ref{sec:blocktop}. 

\subsection{Reduction to toral groups}
We suppose $H_e=\Sigma\times_Z T$, and let $\Hb=H/\Sigma,
\Tb=T/Z$, recalling from Lemma \ref{lem:SigmaTchar} that $\Sigma$ is
normal.  Write $\sub_A(B)$ for the subgroups of $B$
containing $A$.
  \begin{lemma}
\label{lem:subHsubHb}
   Passing to quotients mod $\Sigma$ induces a homeomorphism 
  $$\sub_{\Sigma}(H)/H \cong \sub(\Hb)/\Hb$$
\end{lemma}

\begin{proof}
Certainly there is a bijection $\sub_{\Sigma}(H)\cong \sub 
(\Hb)$, which is a homeomorphism for the Hausdorff metric topology. It 
also preserves cotoral inclusion. Passing to quotients mod $H$ we get 
a homeomorphism as stated.  
\end{proof}

Subgroups close to $H$ (within $\delta$ say), contain $\Sigma$, and therefore we
have natural homeomorphisms
$$\sub^{\delta}(H)/H=\sub^{\delta}_{\Sigma}(H)/H\cong \sub^{\delta}(\Hb)/\Hb, $$
and we have essentially reduced to toral groups as discussed in
\cite{t2wqalg}.

\subsection{Toral groups}
  In the toral case $H/\Sigma=\Hb$ we have the extension 
  $$1\lra \Tb\lra \Hb \lra W\lra 1$$
  where $W=\pi_0(\Hb)=\pi_0(H)=H_d$. 
  The idea is that any subgroup $\Kb\subseteq \Hb$ determines
  $$\Sb=\Kb\cap \Tb\in W\!-\!\sub(\Tb)\cong W\!-\!\sub(\Tb^*). $$
  %  Writing $\Lambda^0=\Tb^*$ and $\Lambda^{\Sb}=(\Tb/\Sb)^*$
    We have an exact sequence
  $$\xymatrix{
    H^2(W;\Sb)\rto&  H^2(W; \Tb)\rto &H^2(W; \Tb/\Sb).
  }$$
  A subgroup $K$ with $\Kb\cap \Tb=\Sb$ exists if and only if the extension
  class $\eps (\Hb)\in H^2(W; \Tb)$ lifts to $H^2(W;\Sb)$, which happens 
  provided $\eps (\Hb)$ is annihilated by the map $\Tb\lra \Tb/\Sb$.  
  Since
  $H^2(W; \Tb)\cong H^3(W; \Lambda_0)$ is a finite group it has an
  exponent $n$ and $\eps(\Hb)$ lifts whenever $\Sb$ contains $T[n]$.
  Accordingly, 
$$\sub(\Hb)/\Hb\lra W\!-\! \sub(\Tb)$$
  is almost surjective, and we may indicate the image by naming the
  class $\eps(\Hb)$ that needs to lift.

  Finally, it is shown in \cite{t2wqalg} that $\Hb$-conjugacy classes
  of lifts are in bijection to $H^1(W; \Tb/\Sb)$.

 In other words, if we choose a small enough neighbourhood of $H$ that 
subgroups all contain $\Sigma$ then we have a Zariski-clopen 
neighbourhood of $H$, which is a space we understand completely as 
the bottom row in the diagram below.

\begin{cor}
  The map 
  $$\sub(H)/H \lra W\!-\!\sub_Z(T)$$
  is almost surjective, in the sense that the image consists of the subgroups 
  $S$ that support $\eps (\Hb)$, and the image contains all
  $W$-invariant subgroups containing $T[n]$ (where $n$ is the order
  of $\eps (\Hb)$). The map has finite fibre  $H^1(W; 
  T/S)$  over a $W$-invariant subgroup $S$. 
\end{cor}

Altogether, the situation can be summarised in the diagram
$$\xymatrix{
    &&\sub^{\delta}(H)/H\ar@{=}[d]&&\\
    0\rto &H^2(W; \Lambda_{\Sb})\rto \dto^{=}& \sub^{\delta}(\Hb)/\Hb\rto^(0.46){\tau} 
    \dto &
    W\!-\!\sub^{\delta}_{\eps(\Hb)}(\Tb)\rto \dto & 0\\
    0\rto &H^2(W; \Lambda_{\Sb})\rto & \sub(\Hb)/\Hb \rto^(0.46){\tau} &
    W\!-\!\sub_{\eps(\Hb)}(\Tb)\rto & 0 
 }$$

\subsection{The earth beneath}
Because of the role of cotoral inclusions in the Zariski topology, we
wish to separate the central torus $T_1$ from the moving torus
$T_f$. This is formalized by mapping to a suitable product. 

\begin{prop}
There is a continuous function 
$$\lambda : \cV^G_H\lra \sub(T_1)_G\times \cN^G_H$$
defined by 
$$\lambda (K):=(\tau (K), \omega (K))$$
where $\omega$ is as in Lemma \ref{lem:omegacts} and $\tau (K)=K\cap T_1$
with $T_1$ is the central torus of $H$. This is almost surjective and and it has finite 
fibres of bounded order. 
\end{prop}

\begin{example} We consider the top block $\cV^H_H$ of a group $H$
  with $H_e=\Sigma\times_Z T$, taken sufficiently small.
 By Lemma \ref{lem:subHsubHb} the space is  essentially independent of
 $\Sigma$,  and homeomorphic to the corresponding space for
 $\Hb=H/\Sigma$. The behaviour is determined by the torus $\Tb=T/Z$.

(a) The trivial case is when $H_e$  is 
semisimple,  $\Tb=1$  and we may take 
$\cN^H_H=\{ H\}$ to be a singleton. It follows that 
$\cV_H^H$ is also a singleton since $H$ has no proper cotoral
subgroups. 

(b) One extreme is  that $H$ has finite centre (which is to say
$H_d$ fixes no circle in $T$). In this case all subgroups in $\cV^H_H$
have finite Weyl group, and hence $\cV^H_H$ is a Stone space (this is
the type of block dealt with in \cite{t2wq}).  

(c) The opposite extreme is when $H_d$ acts trivially on $T$ (so that
$T$ is the identity component of the centre of $H$). In this case $H$ 
is isolated in $\Phi H$ and $\cN^H_H=\{H\}$, and 
$\cV^H_H=\sub(T)_H$ is a Noetherian space, with topology determined 
by the poset of closures of points. 

(d) The general case is the mixed case with $V_1\neq 0$ and $V_f\neq
0$.  An example is worked out in detail in \cite{t2wqmixed}: the group
$H$ is the normalizer  of the maximal torus in $U(2)$. 
In this case $\Sigma$ is trivial, and $T=\Tb$ is the 2-torus with
$H_d$ of  order 2 exchanging the two  factors. The representation $\Lambda_0=H_1(T)$
is the regular representation $\Z [H_d]$. Its rationalization splits
as a sum $\Q[H_d]=\Q\oplus \Qt$.

The fiexed point set $T^{H_d}$ consists of the 
scalar matrices, and in this case it is exactly the centre. 
The space of subgroups 
of the centre is the poset $\fP$ with one 
maximal element $\infty$ (corresponding to the centre $T^{H_d}$
itself) and one minimal element $n$ for each cyclic subgroup of order 
$n\geq 1$. 

We may take the space $\cN^H_H$ to consists of $H$ itself, together
with the subgroups 
$K_{2n}$ containing the centre  and matrices $\diag (a, a^{-1})$ with 
$a^{2n}=1$. Thus $\cN^H_H$ homeomorphic to the one-point compactification of the 
positive integers.

The map 
$$\lambda : \cV^H_H\lra \fP \times \cN^H_H $$
is actually surjective. The fibres over the points $(\infty,n)$ are 
singletons 
The fibres over the points $(m, \infty)$ with $m<\infty$  are pairs. 
The fibres over the points $(m, n)$ with $m, n<\infty$  have two or 
three elements. There are three elements precisely when $m+n, m-n$ are 
both even.  
\end{example}

  \subsection{From $H$-conjugacy to $G$-conjugacy}
We have explained that the block $\cV^G_H$ is constructed by finding a 
neighbourhood $\cN^H_H$ of $H$ in $\Phi H$ small enough that subgroups also have 
finite $G$-normalizer and then taking the closure under cotoral 
specialization. If we just specify $\cN^H_H$ as a metric ball of 
radius $\delta$ this gives 
$$\xymatrix{
  \sub_f^{\delta}(H)/H\ar@{=}[r]&
  \cN^H_H\rto^{\pi} \dto &\cN^G_H\dto\\
\sub^{\delta}(H)/H\ar@{=}[r]&  \Lambda_{ct}\cN^H_H\rto^{\pi}\ar@{=}[d] &\Lambda_{ct}\cN^G_H\ar@{=}[d] \\
  &\cV^H_H\rto &\cV^G_H 
}$$

\begin{lemma} The maps labelled $\pi$ are surjective and have finite fibres of 
  bounded size. They are also both continuous and closed. 
\end{lemma}

\begin{proof}
The maps $\pi$ are induced from the inclusion $\sub(H)\lra \sub(G)$ by 
passing to conjugacy, and are therefore continuous by the universal property of the quotient 
topology. Since all spaces are compact and Hausdorff, closed sets 
coincide with compact sets, so continuous maps are closed. 

The diagram 
$$\xymatrix{
  \cV^H_H\rto^(0.3){\lambda}\dto &\sub(T)_H \times \cN^H_H\dto \\
  \cV^G_H\rto^(0.3){\lambda}  &\sub (T)_G \times \cN^G_H 
  }$$
commutes since the subgroups represented in $\cN^H_H$ also have finite
$G$-Weyl groups by Lemma \ref{lem:PhiHPhiG}. The fibres of the top row are 
bounded from our analysis of the top block $\cV^H_H$, and the fibre of 
the left hand vertical is bounded by the order of the Weyl group of $G$. 
  \end{proof}

\begin{remark}
We may wish to be more precise. The map $\cN^H_H\lra 
\cN^G_H$ is surjective by definition, so we 
just need to ask when it happens that $(K)_G=(K')_G$ but $(K)_H\neq 
(K')_H$. We suppose $K'=K^g$. 

We suppose that $K$ is a full subgroup of $H$. There seem to be four 
different levels of conjugacy, presumably getting less and less rare. 

\begin{itemize}
\item $K^g$ is not full. In other words $K^g$ does not represent all 
  components of $H$. This means  $g$ conjugates some elements from 
  non-identity components into $\T$. This means that $d(K,K^g)$ is at 
  least the distance 
  between two components of $H$, and hence there is a neighbourhood of 
  $e$ not containing any element $g$.  To see this happens we need 
  only consider the fact that the subgroups  $D_2$ (generated by a 
  reflection) and $C_2$ (generated by a rotation) in $O(2)$ become 
  conjugate in $SO(3)$. 
\item $K^g$ is full but $S'=K^g\cap \T\neq K\cap \T$. 

  \item $K^g$ is full and $S=S'$ but $g\not \in N_G(H)$
\item $K^g$ is full and $S=S'$ and $g \in N_G(H)$. This means that $g$
  acts on $H^1(W; \T/S)$
  \end{itemize}
\end{remark}

\section{Standard neighbourhoods}
\label{sec:blocktop}
The purpose of this section is to relate the topology on $\cV^G_H$ to
group theory by describing a base of  neighbourhoods determined by
product groups. We introduce notation but there are no surprises. 

\subsection{Standard neighbourhoods}

We will describe a topological basis of standard neighbourhoods $\cN^H_H$ of $H$ in
$\Phi H$. We start with special cases and then increase the
complexity.

  First consider the case $H=T$.    For  $1\leq A \leq \infty$, we may 
  consider the set    $\cN_T^{all} (A)$ of subgroups 
 containing a subgroup $T[a]$ with $a\geq 
 A$, with the convention that $T[\infty]=T$,  and 
recognize that not every subgroup is of the form $T[a]$ (the
 superscript `all' reflects the fact we are not restricting to subgroups with
 finite Weyl group).  The sets $\cN_T^{all}(A)$ give a base of
 neighbourhoods of $T$ inside $\sub (T)$.

 Next we suppose $H_e$ has trivial semisimple component, so $H_e=T$,
 but take into account the fact that  $H_d$ acts on $T$ 
 and we have a decomposition of a
 finite cover of $T$ into a product, giving a finite cover
 $$ \Tt_1\times \cdots \times \Tt_s \lra T, $$
 with kernel $L$. We write
  $$\Tt[a_1, \ldots , a_s]=\Tt_1[a_1]\times \cdots \times \Tt_s[a_s], $$
and  $T [a_1, \ldots , a_s]$ for the image in $T$.  If  $1\leq A_i\leq \infty$, we may 
  consider the set    $\cN_H^{all} (A_1, \ldots, A_s)$ of subgroups
  containing a subgroup $T[a_1, \ldots , a_s]$ with $a_i\geq  A_i$.
These subgroups $T[a_1, \ldots ,a_s]$ give a base of neighbourhoods of $T$ inside $\sub (T)$. If we wish to restrict to subgroups with finite
Weyl group we need to restrict to those with $a_1=\infty$, and we take
$$\cN_H( A_2, \ldots , A_s)=\cN_H^{all}(\infty, A_2,
\ldots , A_s)\subseteq  \Phi H. $$

Finally, in the general case that the semisimple component $\Sigma$  of $H_e$ is non-trivial we take
$$\cN_H( A_2, \ldots , A_s)=\{ K \st K_{\e}=\Sigma \times_Z S \mbox{
  with }
Z_T+T[\infty, a_2, \ldots ,a_s]\subseteq S \mbox{ for some } a_i\geq A_i\}.$$

  \begin{lemma}
The neighbourhoods $\cN_H^{all} (A_1, \ldots , A_s)$ give a base of 
neighbourhoods of $H$ in the Hausdorff metric topology on $\sub (H)/H$.

The neighbourhoods $\cN_H (A_2, \ldots , A_s)$ give a base of 
neighbourhoods of $H$ in $\Phi H$.
\end{lemma}

\begin{proof} The second statement follows from the first by the
  calculation of normalizers, so we discuss the first. 

First, provided the distances are less than the distance between
components of $H$, $d(A,B)=d(A_{\e}, B_{\e})$ and we may restrict
attention to identity components. Next, 
we note that by a suitable choice of metric $d(\Sigma\times A,
\Sigma\times B)=d(A,B)$, and the quotient map $\Sigma \times T\lra
\Sigma\times_ZT$ is a local isometry. This means it suffices to assume
$H_e=T$. Again, the map $\Tt \lra T$ is a local isometry so we may
argue in $\Tt$.

Now scale so that the distance from a circle to a subgroup 
of order $n$ is $1/n$, we see that 
$$1/\max(a_i)\leq d(\Tt[a_1]\times \cdots \times \Tt[a_s],\Tt)\leq
1/\min (a_i). $$

We must show two containments.  First we show that for any $\epsilon >0$, a ball of radius $\epsilon$ around $H$ contains one of the $\cN_H(A_1, \ldots , 
A_s)$. Indeed, if $a_i\geq 1/\epsilon $ for all $i$, then 
$d(\Tt[a_1] \times \cdots \times \Tt[a_s], \Tt) <\epsilon$. 
Since any subgroup containing
$\Tt[n_1]\times \cdots \times \Tt[n_s])$ is even closer to $H$, 
$B_\epsilon (H)$ contains
the set $\cN_H(A_1, \ldots , A_s)$ if we ensure $A_i>1/\epsilon$ for
all $i$.

In the other direction, $\Gamma \not \supseteq \Tt[a_1]\times \cdots
\times \Tt[a_s]$ means $d(\Gamma, \Tt)>1/2\max(a_i)$,  and hence 
$B_{1/2A}(\Tt)\subseteq \cN_H(A_1, \cdots , A_s)$ where  
 (where $A=\max (A_i)$).
    \end{proof}

\subsection{The map $\lambda$ is almost surjective}
We have shown that a neighbourhood $\cN^H_H$ of $H$ in $\Phi H$
 can be chosen so that its image $\cN^G_H$ lies in $\Phi G$. We now
 observe that we can use one of the standard neighbourhoods and ensure
that $\lambda$ is almost surjective. 

\begin{lemma}
  We may choose $\cN_H^H=\cN_H(A_2, \ldots , A_s)$ and $D$ so that 
  the image of   $\lambda : \cV^G_H\lra \sub(T_1)_G\times \cN^G_H$
contains $\sub_{T_1(D)} (T_1   )_G\times \cN^G_H$. 
  \end{lemma}

  \begin{proof}
  In a group $H$, with $H_e=\Sigma \times_Z T$, any full subgroup close to 
  $H$ contains $\Sigma$. Similarly, any subgroup $K$ close to $H$ will 
have $N_G(K)\subseteq N_G(H)$. We assume $\cN_H$ is chosen small 
enough in both senses.

Subgroups of $H_e$ correspond to subgroups of 
$\tilde{H}_e:=\Sigma \times T$ containing $Z$, and so any one containing $\SSi$ is of 
the form $\Sigma \times C$ with $C$ containing the image $Z_T$ of $Z$ in $T$. 

To see that $\lambda$ is surjective we need to show that if $\Khat \in 
\cN^G_H$ and $A\subseteq T_1$ then $(\Khat , A)$ is 
in the image. We note that since $\Khat$ has finite Weyl group, 
$\Khat\supseteq T_1$, so that $\Khat_{\e}=\Sigma \times_Z (\Tt_1\times 
B)/L_f$, and $\Khat= H ((\Tt_1\times B)/L_f, \sigma)$ for some splitting 
$\sigma$. We choose a finite subgroup 
$\Ct_{fin}$ with 
$$L_f \subseteq \Ct_{fin} \subseteq  \Tt_1\times \Tt_f$$
supporting the extension. Now provided $\Ct_{fin}\subseteq A\times B$
we  take $K=H((A \times B)/L_f, \sigma))$ and have $\lambda (K)=(\Khat, 
A)$. If we suppose $\Ct_{fin}$ is of exponent $D$ we need only show 
that $A\times B$ contains $T[D]$. 
By requiring $\cN_H$ to be small enough we can ensure that $B$
contains all points of order dividing $D$ in $T_f$, so that we only 
need to ensure that $A$ contains $T_1[D]$.

To see the fibres are finite we note that given $\Khat$, this 
specifies $\Sigma, Z$ and $B$ and we then make a choice of splitting $\sigma$ so 
that $\Khat = H((\Tt_1\times B)/L_f, \sigma)$, and a choice of a 
finite subgroup supporting the extension. There are only finitely many 
subgroups $C^*\subseteq (\Tt_1 \times B)/L_f$ so that $C^*\cap T_1=A$
with $K=H(C^*, \sigma)$. However it is possible that there are other 
splittings $\sigma'$ so that $K'=H(C^*, \sigma')$ has the same image 
under $\lambda$ the finiteness in number comes from the fact that the 
extension is supported on a finite subgroup. Accordingly 
$\sigma$ can be chosen to map into a finite subgroup, and by 
finiteness of $H_d$ there are finitely many choices. 
\end{proof}

\section{The topology of the top block from the toral lattice}
\label{sec:topology}
We are now equipped to describe the topology on the top block $\cV^H_H$ in
terms of homological data.  For other blocks $\cV^G_H$, the topology
is given by the quotient map $\cV^H_H\lra \cV^G_H$. 

\subsection{Flattening}
We start with a subgroup $H$, and assume the neighbourhood $\cN^H_H$
is small enough to ensure all subgroups $K$ contain $\Sigma$. It
follows that the same is true for $\cV^H_H$ and hence $\cV^H_H\cong
\cV^{\Hb}_{\Hb}$. Thus we are effectively reduced to the  toral group
$\Hb$, and we will write $\Tb$ for the identity component.

The basic data is the map 
$$\nu: \cV^H_H\lra W\!-\!\sub(T)_H\cong W\!-\! \sub(\Tb)_H, $$
taking $K$ to $\Sb=(K/\Sigma) \cap \Tb$. Associated to each subgroup $K$ we also have 
the $\Z W$-module $\mu (K)=\Lambda_{\nu (K)}=H_1(T/\nu(K))$.

To specify the topology we need to be able to tell when sequences 
converge. We therefore suppose given full subgroups $K_1, K_2, \ldots, 
K_*$ of $H$ and ask if $K_i\lra K_*$.  Since only finitely many lattices
occur as submodules $\Lambda^{\Sb}$ of $\Lambda^0$ we may suppose that up to
isomorphism $\Z W$-module $\mu(K_i)$ is independent of $i$.
Associating $K/\Sigma$ to $K$ is continuous, and since $\Tb$ is
characteristic, it follows that if $K\lra K_i$ then  $\Sb_i\lra \Sb_*$.
If $\Sb_*$ is the limit, omitting initial terms if necessary, by
Montgomerry-Zippin we may suppose 
$\Sb_i\subseteq \Sb_*$, so that there are induced maps 
$$\alpha^i_*: H^2(W; \Sb_i)\lra 
H^2(W; \Sb_*) \mbox{ and } \beta^i_*: H^1(W; \Tb/\Sb_i)\lra H^1(W; 
\Tb/\Sb_*). $$
Since $K_i\subseteq  K_*$ we have $\alpha^i_*(\eps (K_i))=\eps (K_*)
$. 

\subsection{Ramification and degeneration}
We describe the ramification sequence of \cite[Section 3]{t2wqalg}
for $K_i$; the case of $K_*$ is precisely similar. The $H$-conjugacy classes
of subgroups $K_i$ associated to $\Sb_i$ and a fixed extension class
$\eps_i\in H^2(W; \Sb_i)$ lifting $\eps (H)$ are parametrised by
$H^1(W; \Tb/\Sb_i)$ (note that this parametrises $H$-conjugacy classes
and not extension classes).  More precisely,
once we have chosen a splitting $\sigma_i: W\lra \Tb$ whose factor set
takes values in $\Sb_i$ and represents $\eps_i$, all other splittings  are $\sigma_i \cdot g$
for a function $g: W\lra \Tb$ which give an element of $Z^1(W;
\Tb/\Sb_i)$.  The subgroups  are
$H$-conjugate if the cycles  differ by a boundary. For $x\in H^1(W;
\Tb/\Sb_i)$, we  will write $\kappa_i (x)$ for the conjugacy class
represented by $\sigma_i \cdot g$, where $g$ is a cocycle
representative for $x$.  

We have a  diagram
$$\xymatrix{
H^1(W; \Tb/\Sb_i)\rto \dto^{\beta^i_*} &H^2(W; \Sb_i)\rto \dto^{\alpha^i_*} &H^2(W;\Tb)\dto^= \\
H^1(W; \Tb/\Sb_*)\rto  &H^2(W; \Sb_*)\rto &H^2(W;\Tb) 
}$$
The extension classes $\eps (G), \eps_i, \eps_*$ are mapped to
each other. Choose a splitting $\sigma_i$ for $K_i$, which also gives
a splitting for $K_*$ and the same formula applies for
$\kappa_i $ for the two groups $K_i$ and $K_*$. This shows 
that if $K_i$ is is represented by $\kappa_i(x)$ in $H^1(W; \Tb/\Sb_i)$
 then $K_*$ is represented by $\kappa_i(\beta^i_*([x]))$ in $H^1(W;
 \Tb/\Sb_*)$, and the degeneration is given by $\beta^i_*$.
$$\xymatrix{
x^{\sigma_i}_{K_i} \dto^{\beta^i_*} &\eps_i \rto \dto^{\alpha^i_*}
&\eps (H)\dto^= \\
\beta^i_*(x^{\sigma_i}_{K_i})&\eps_*\rto &\eps(H) 
}$$

\begin{lemma}
  Fix an integral representation $M$ of $W$ occurring as a submodule of 
  $\Lambda^0$.    Suppose  $K_i, K_*\in 
  \sub_{\Sigma}(H)$ correspond to  $S_i, S_*\in \sub_Z(T)$ and
  $x_{K_i}, x_{K_*}$ for corresponding splittings.  
We suppose $\Lambda^{\Sb_i}\cong M$
  (independent of  $i$) and  $K_i\subseteq K_*$. Convergence of the 
  subgroups $K_i$ is determined by the condition   $K_i\lra K_*$ if and only if 
\begin{itemize}
\item  $\Sb_i\lra \Sb_*$
\item  $\alpha^i_*(\eps_i)=\eps_*\in H^2(W; 
  \Tb)$ 
\item $\beta^i(x^{\sigma_i}_{K_i})=x^{\sigma_i}_{K_*} \in H^1(W; \Tb/\Sb_*)$
\end{itemize}
\end{lemma}

  \begin{proof}
We have already argued that the condition is necessary. Suppose then 
given $K_i, K_*$ satisfying the three conditions. By the first 
condition, discarding initial terms if necessary we may suppose 
$\Sb_i\subseteq \Sb_*$ and $\Sb_*$ is the closure of $\bigcup_i \Sb_i$. 
By construction the third condition shows that $K_i=H(\Sb_i, \sigma_i), K_*=H(\Sb_*,
\sigma_i)$.
    \end{proof}
    \subsection{Splitting}
We now recognize that in the toral lattice, the trivial repesentation behaves quite
differently to the others, and we split the flattened block into a
product of two pieces accordingly. 
The block of subgroups $K$ is flattened into considering the 
 subgroups  $\Sb=K/\Sigma \cap \Tb$, and then we use the fact that this such subgroups
 almost split    as a product of a central part and a moving part.

    \begin{prop}
      There is a commutative diagram
      $$\xymatrix{
        &\sub(\Tb)_G=\sub_{L_f}(\Tb_1\times \Tb_f)_G\drto \\
        \cV^G_H\urto^{\nu}\drto_{\lambda}&&
\sub_{L_f}(\Tb_1)_H\times \sub_{L_f}(\Tb_f)_H\\
        &\sub_{L_f}(\Tb_1)_G \times \cN^G_H\urto_{\{id, \nu \} }
        }$$
        where $\lambda(K)=(\tau (K), \omega(K))$ with $\tau (K)=K/\Sigma\cap
        \Tb_1$. We write $\sigma$ for the upward composibte so that 
        $$\sigma (\Khat, \Rb)=(\Khat\cap \Tb_f+\Rb).$$
      \end{prop}

      \begin{proof}
By the top route, a subgroup $K$ is sent to $(K\cap \Tb_1, K/\Sigma \cap \Tb_f)$. By the bottom route, it is sent first to $(K/\Sigma \cap
\Tb_1, \omega (K))$, so it remains only to point out that $\nu (\omega (K))=(K/\Sigma\cap
\Tb+\Tb_1)$, which meets $\Tb_f $ in $K/\Sigma \cap \Tb_f$ as claimed. 
              \end{proof}

We note that we have a good understanding of the fibres  of these
maps. The fibres of $\nu$  are of the form $H^1(W;\Tb/\Sb)$, and
similarly for the restricted $\nu$ on the bottom  route. the second
map on the upper route is is addition of subgroups under
$$\Tb_1\times \Tb_f\lra \Tb, $$
so that subgroups in the fibre are distinguished by how much of the
finite group $L_f$ they contain.

\section{Arithmetic Grassmannians}
\label{sec:GrassQ}
 \subsection{Reduction to rationally isotypical modules}
 Continuing, up to finite error, we show that we may decompose $W\!-\!\sub(\Lambda^0)$
 as a product of spaces $W\!-\!\sub(\Lambda_i)$ where $\Lambda_i\tensor
 \Q=S_i^{\oplus m_i}$ is a sum of simple modules.

 \begin{lemma}
   Suppose $\Lambda_1, \Lambda_2\subseteq \Lambda$. We may consider 
  the map 
  $$\delta: \Gr_*(\Lambda)\lra \Gr_*(\Lambda_1)\times \Gr_*(\Lambda_2)$$
taking $\delta(\Gamma)=(\Lambda_1 \cap \Gamma, \Lambda_2 \cap 
\Gamma)$. If $\Lambda\tensor \Q$ is spanned by $\Lambda_1, \Lambda_2$
then the map $\delta$ has finite fibres since 
$$\Lambda_1\cap \Gamma +\Lambda_2\cap \Gamma \subseteq \Gamma $$
and this is a rational isomorphism. 

If $\Lambda_1\cap \Lambda_2=0$ then the map $\delta$ is surjective, 
since $(\Gamma_1, \Gamma_2)$ is the image of  $\Gamma_1+\Gamma_2$.
\end{lemma}

We now suppose
$$\Lambda^0\tensor \Q=S_1^{\oplus m_1}\oplus S_2^{\oplus m_2}\oplus
\cdots \oplus S_t^{\oplus m_t}$$
and observe that if we take $\Lambda_i=\Lambda^0\cap S_i^{\oplus m_i}$
we have $\Lambda_i\cap \Lambda_j=0$ for $i\neq j$ and
$$\Lambda_1\oplus \cdots \oplus \Lambda_t\subseteq \Lambda^0$$
is a rational isomorphism.

\begin{cor}
 There is a surjective map 
$$W\!-\!\sub(\Lambda^0)\lra W\!-\!\sub(\Lambda_1)\times
W\!-\!\sub(\Lambda_2)\times \cdots \times W\!-\!\sub(\Lambda_t)$$
with finite fibres.
\end{cor}
 
This has finally reduced us to looking at the spaces
$W\!-\!\sub(\Lambda)$.

\subsection{Integral Grassmannians}

Starting without a group action, we have the Stone space
  $\sub(\T)=\Gr_*(\Lambda)$, where
  $\T\cong T^r$ is a product of $r$ circles and $\Lambda=\T^*\cong
  \Z^r$ is a free abelian group of rank $r$. The topology on this can
  be described by a metric, and there is a filtration by rank of
  subgroups
  $$0=\Gr_{\geq r}(\Lambda)\subset \Gr_{\geq r-1}(\Lambda)\subset
  \cdots \subset \Gr_{\geq 1}(\Lambda)\subset \Gr_{\geq 0}(\Lambda)
  =\Gr_*(\Lambda)$$ 
Each term is open and dense in the whole space, and the pure strata
are discrete.

It is natural to divide the submodules according to their
rationalization, so we consider the map $\Gr_*(\Lambda)\lra
\Gr_*(\Lambda\tensor \Q)$.
The map is surjective since one can look at a basis of a submodule in 
terms of an integral basis of $\Lambda$ and clear denominators. 
The fibre over an
  $i$-dimensional representation $V$ consists of sublattices of $V\cap
  \Lambda \cong \Z^i$ and is in bijection with $GL_i(\Q)[\Z]/GL_i(\Z)$
  of integral matrices with non-zero determinant modulo those with
  determinant $\pm 1$.

This is generally fairly complicated. To start with, the fibre has a
free action of the monoid $\PP$ of positive integers $\PP$, since 
if $\Lambda'\subseteq \Lambda$ the multiple $m\Lambda'$ is another 
$W$-submodule. If $\Lambda $ is of rank 1 the fibre will be a single
orbit, but usually it will be much more complicated. There will be a
component for each $\Z W$-isomorphism class of lattices $\Lambda'$, 
and then a copy of $\Mono (\Lambda')/\Aut (\Lambda')$ for each one. 
For example if $W$ is of order 2 and $\Lambda =\Z W$ the component 
consisting of lattices with the same rationalization as $\Lambda$
consists of all $\Lambda_1(m,n)$ (with $m,n\geq 1$) and all
$\Lambda_2(m,n)$ (with $m,n\geq 1$ and $m+n, m-n$ even). Thus there 
are two components and each of them has two parameters.

  \subsection{Integral equivariant Grassmannians}
  If $\Lambda$ is an integral representation of 
  the finite group $W$, we may consider the subspace 
  $\Gr_*^W(\Lambda)\subseteq \Gr_*(\Lambda)$ of $W$-invariant 
  subspaces. This is compatible with rationalization
  in that there is a square
  $$\xymatrix{
    \Gr_*^W(\Lambda)\rto\dto & \Gr_*(\Lambda)\dto \\
    \Gr_*^W(\Lambda\tensor \Q)\rto & \Gr_*(\Lambda\tensor \Q) \\
  }$$
  
\subsection{Rationality}
  If $\Lambda\tensor K\cong S^{\oplus m}$ is a sum of copies of a single
  simple module $S$ of dimension $d$, then all subspaces
  are sums of copies of $S$ and $\Gr_{i}^W(S^{\oplus m})=\emptyset$ unless
  $i=dj$. If in addition $\End_K(S)=K$ then $\Gr_{dj}(S^{\oplus m})=\Gr_j(K)$.
  In general we easily see that the action of $GL^W(S^{\oplus m})$ is
  transitive on submodules isomorphic to $S^j$ and hence 
  $$\Gr_{dj}^W(S^{\oplus m})=GL^W(S^{\oplus m})/[GL^W(S^{\oplus j}\oplus 0)\times GL^W(0\oplus
  S^{\oplus (m-j)})]. $$
 By Schur's Lemma, 
$$GL^W(S^{\oplus m})\cong GL_m(\cE_K(S))$$
where $\cE_K(S)=\End_{KW}(S)$ is a division algebra. 
  If $K$ is an algebraically closed field, $\End_K(S)=K$.

 \subsection{The structure of a block}

 Each block is up to finite error a product of copies of standard
 pieces.

 \begin{prop}
   If $H_e=\Sigma\times_Z T$ with $\Lambda_0=H_1(T)$, and  if
   $$\Lambda_0\tensor \Q =S_1^{\oplus m_1}\oplus S_2^{\oplus
     m_2}\oplus \cdots \oplus S_s^{\oplus m_s}, $$
   and if $\Lambda_i =S_i^{\oplus m_i}\cap \Lambda_0$ then there is a map
 $$\nu':\cV^G_H\lra \sub(T_1)_G\times \Gr_*^W (\Lambda_2)\times
 \Gr_*^W (\Lambda_3)\times \cdots \times  \Gr_*^W (\Lambda_s). $$
 which is almost surjective and which has finite fibres of bounded
 size. There are maps
 $$\Gr_n^W(\Lambda )\lra \Gr_n(\Lambda\tensor \Q)$$
 The fibre breaks into pieces according to the $\Z W$-forms of a
 rational subspace. Finally, the space $\Gr_*^W(\Lambda)$ is a
 compactification of the space of full rank $W$-submodules of $\Lambda$.
\end{prop}

\section{A normalizer variation lemma}
\label{sec:normalizers}
 We  are most interested in controlling the variation of the Weyl groups $W_G(K)=N_G(K)/K$. We will do this by
controlling the variation of the normalizer $N_G(K)$.
We've already see the variation is
not functorial in $K$ and we give some examples to show the variation is not simple. 
However we will eventually prove the critical fact that taking normalizers is 
uppper semi-continuous. More precisely, every subgroup $H$ has a 
neighbourhood so that for subgroups $H_i$,  $d(H_i, H)<\eps$ implies 
$N_G(H_i)\subseteq N_G(H)$.

\begin{example} The normalizer $N_G(K)$ is not continuous as a
  function of $K$. This example shows this is because of the
  behaviour of components. For example we may choose $G=T^2\sdr C_3$,
  with $C_3$ acting non-trivially.  We may choose a succession of
  finite subgroups $F$ of the 2-torus $T^2$. If  $F$ is
  invariant under $C_3$ the normalizer is $G$, and otherwise it is
  $T^2$. If we choose a sequence $F_i$ tending to $T^2$ with the even
  terms invariant under $C_3$ and the odd ones not then the sequence
  $N_G(F_i)$ is not convergent (for example $F_{2i}=C_{2i}\times
  C_{2i}$, $F_{2i+1}=C_{2i}\times C_{2i+1}$). 
\end{example}

The bad behaviour is all about the component group, so may wish to
consider  the identity component $N_G^e(K)$ instead. However $N_G^e$ 
 is not continuous either. 

\begin{example}
We take $G=O(2)$, and consider the sequence of dihedral subgroups
$D_{2n}$. Of course $N_{G}(D_{2n})=D_{4n}$ so $N_G^e(D_{2n})$ is
trivial for all $n$, but $N_G^e(O(2))=SO(2)$.  
\end{example}

\begin{thm}
\label{thm:normalizersusc}
Taking normalizers is upper semicontinuous in the sense that if
$H_i\lra H$ (for subgroups $H_i$ of $H$),  then there is an $n$ so that $N_G(H_i)\subseteq N_G(H)$
for $i\geq n$. 
\end{thm}

\begin{remark}
If we do not know that $H_i$ is a subgroup of $H$, we need to permit
conjugation. The Montgomery-Zippin Theorem shows that $H_i$ is
subconjugate to $H$, and two different subconjugations will differ by
an element of $N_G(H)$. Elements of $N_G^f(H)$ will give homotopic
identifications. Since $W_G^d(H)$ is finite, we may suppose by making
$n$ big enough that we are restricted to $N_G^f(H)$, and hence we will
get well defined cohomology maps. 
\end{remark}

\begin{proof}
Let us suppose $H^e=\Sigma\times_Z T$, with maximal torus $T\Sigma
\times_Z T$. Suppose also that
$H_i^e=\Sigma_i\times_{Z_i}T_i$. Conjugating by elements of $H$ we may
assume that $T\Sigma_i \times_{Z_i} T_i\subseteq T\Sigma \times_Z T$. 
First we note that $\Sigma_i\subseteq \Sigma$, and $\Sigma_i\lra
\Sigma$. Discarding initial terms we may suppose $\Sigma_i=\Sigma $
and $Z_i=Z$. Now 
$$N_G(H_i)\subseteq N_G(H_i^e)\subseteq N_G(\Sigma), $$
and so we may replace $G$ by  $N_G(\Sigma)$ and suppose $\Sigma$ is 
normal in $G$. Now we may factor out $\Sigma$, and it suffices to show 
$N_G(H_i/\Sigma)\subseteq N_G(H/\Sigma)$. To simplify notation we 
suppose $\Sigma=1$.

Now we have $H_i^e=T_i$ and $H^e=T$ and we  argue
$$N_G(H_i)\subseteq N_G(H_i^e) \subseteq N_G(H^e), $$
by Lemma \ref{lem:torusnormalizers}. Hence we may replace $G$ by
$N_G(H^e)$ and suppose $T$ is  normal in $G$. Thus 
$$H_i/H_i\cap T \lra H/T , $$
and since $H/T$ is finite, we may discard initial terms and suppose
$$H_i/H_i\cap T = H/T . $$
Note that $H$ is a union of cosets $kT$ where $k\in H$, and since
$H/T=H_i/H_i\cap T$ we may suppose $k\in H_i$. 
Now we see that if $g\in N_G(H_i)$ then 
$$g(kT)g^{-1}=gkg^{-1}gTg^{-1}=k'T $$
where $k'\in H_i$. Thus $g$ takes all cosets in $H$ to other cosets in
$H$ and $g\in N_G(H)$ as required. 
\end{proof}

\begin{lemma}
\label{lem:torusnormalizers}
If $T$ is a torus in $G$ and $T_i\lra T$ with $T_i\subseteq T$  then 
there is an $n$ so that $N_G(T_i)\subseteq N_G(H)$ for $i\geq n$.  
\end{lemma}

\begin{proof}
It suffices to show that if $i$ is large enough and $g$ normalizes $T_i$ then it normalizes $T$. Since $T$ is connected, it is generated by a neighbourhood of its identity, and we can work in an open ball of radius $2\delta>0$ so that the exponential map is a homeomorphism $B_{LG}\cong B_G$ between the balls. Thus it suffices to show that $ad(g)$ takes the tangent space $LT$ to $LT$. If $T$ is of rank $s$ and we choose linearly independent elements $v_1, \ldots , v_s$ in $ LT$, since $ad(g)$ is a linear isomorphism, their images will again be linearly independent in $LG$. It therefore suffices to show the $v_s$ can be chosen so that $ad (g) (v_1), \ldots ad(g) (v_s)$ lie inside $LT$.  

For this we will use the exponential map, and the fact that it is easy to understand within a torus. Since $ad(g)$ is the derivative of $c_g$ at the identity, and since every automorphism of a torus is linear, we have the following commutative diagram of abelian groups 
$$\xymatrix{
&T^g&&LT^g\ar[ll]^{exp}\\
T\ar[ur]^{c_g}&&LT\ar[ur]^{ad(g)}\ar[ll]^{exp}&\\
&T_i^g\ar[uu]&&LT_i^g\ar[ll]^{exp}\ar[uu]\\
T_i\ar[ur]^{c_g}\ar[uu]&&LT_i\ar[ur]^{ad(g)}\ar[ll]^{exp}\ar[uu]&\\
}$$
We note that the exponential maps are local homeomorphisms within the balls $B_{V}=B_{LG}\cap V$ for vector subspaces $V$ of $LG$. If $g$ normalizes $T_i$ then $T_i^g=T_i$

Choose coordinates and a metric on $LT$ so that it has orthonormal basis $e_1, \ldots , e_s$ (top right front). By averaging we may assume $ad(g)$ is an isometry.  
We may find $\epsilon>0$ so that if $v_s\in LT$ is within $\epsilon $ of $\delta \cdot e_i\in B_{LT}$, then the $v_s$ are linearly independent (any $\epsilon <\delta/2$ will do), and choose $N$ so that $d(T_i,T)<\epsilon$ for $i\geq N$. Thus there is a point of $T_i$ within $\epsilon$ of every point in $B_T$.

We may therefore choose points $x_s\in B_T\cap T_i$ so that they (or more properly their counterparts $v_s$ in $B_{LT}\cap LT$) form a basis of $LT$. Since the exponential map is surjective, we may find $w_s\in LT_i$ so that $x_s=exp(w_s)$ (of course $w_s$ will usually not be inside $B_{LT}$, but since $w_s$ and $v_s$ map to $x_s$ under the exponential map, they differ by the lattice $\ker (exp: LT\lra T)$). 
Now if $g$ normalizes $T_i$ we have $T_i^g=T_i$ and so $ad(g)(w_s)\in LT_i$, and hence 
$$exp (ad(g)v_s)=exp (ad(g)(w_s))\in exp(LT_i)\subseteq exp(LT)$$ as required.  A priori we only know $ad(g)(v_s)\in B_{LG}$, but  $exp$ is a homeomorphism on $B_{LG}$ and the images lie in $LT$, so it follows that $ad(g)(v_s)\in B_{LG}\cap LT$ as required. 
\end{proof}

It is often useful to know that the identity component of the
normalizer is actually continuous. 
\begin{cor}
If $H_i\lra H$ then $N_G(H_i)_e\lra N_G(H)_e$. 
\end{cor}

\begin{proof}
We have proved $N_G(H_i)\subseteq N_G(H)$, and this implies a
containment of identity components. It remains to show that every
element of $N_G(H)^e$ normalizes $H_i$. 
  
As in the original proof, we may assume that $H_i$ and $H$ share the
semisimple factor of the identity component and work inside its
normalizer. Factoring out, we may assume that $H_i$ and $H$ are
toral. It suffices to show that for any toral group $H$, the
automorphism group  $\Aut (H)$ is discrete. 
  \end{proof}

%\bibliographystyle{plain}
%\bibliography{../../jpcgbib}
\end{document}